\documentclass[a4paper,12pt]{amsart}
\usepackage{comment,amsmath,amssymb,amsthm,changepage,pifont,
            graphicx,soul,xcolor,longtable,mathtools}
\usepackage[autostyle]{csquotes}
\usepackage[english]{babel}
\usepackage[foot]{amsaddr}
\allowdisplaybreaks

\newtheorem{theorem}{Theorem}[section]

\newtheorem{lemma}[theorem]{Lemma}

\theoremstyle{remark}

\newtheorem{definition}[theorem]{Definition}

\theoremstyle{theorem}

\newcommand{\ZZ}{\mathbb{Z}}
\newcommand{\RR}{\mathbb{R}}
\newcommand{\QQ}{\mathbb{Q}}

\newcommand{\CC}{\mathbb{C}}

\newcommand{\NN}{\mathbb{N}}

\newcommand{\cM}{\mathcal{M}}
\newcommand{\cO}{\mathcal{O}}
\newcommand{\cL}{\mathcal{L}}

\DeclareMathOperator{\graph}{graph}
\DeclareMathOperator{\radop}{rad}
\DeclareMathOperator{\rv}{rv}
\DeclareMathOperator{\Th}{Th}

\newcommand{\tr}{\mathrm{tr}}
\newcommand{\RV}{\mathrm{RV}}
\newcommand{\val}{\mathrm{val}}
\newcommand{\acl}{\mathrm{acl}}
\newcommand{\dcl}{\mathrm{dcl}}

\usepackage[colorinlistoftodos]{todonotes}
\usepackage[colorlinks=true, allcolors=blue]{hyperref}

\title{Counting rational points on transcendental curves in valued fields}

\author{Floris Vermeulen}
\address{Department of Mathematics, Univeristy of M\"unster, Germany}
\email{florisvermeulen.math@gmail.com}

\date{}

\begin{document}

\begin{abstract}
We prove upper bounds on the number of rational points on transcendental curves in arbitrary $1$-h-minimal fields, similar to the Pila--Wilkie counting theorem in the o-minimal setting.
These results extend results due to Cluckers--Comte--Loeser from $p$-adic fields to arbitrary valued fields of mixed characteristic.
Our methods rely on parametrizations, where we avoid the usage of $r$-th power maps, combined with the determinant method.
\end{abstract}

\maketitle

\section{Introduction}

Recall the statement of the Pila--Wilkie counting theorem, which tells us that transcendental sets do not contain many rational points of bounded height.

\begin{theorem}[{{{Pila--Wilkie~\cite{PW}}}}]
Let $X\subset \RR^n$ be definable in an o-minimal structure. 
Then for every $\varepsilon > 0$ there exists a constant $c = c_\varepsilon > 0$ such that for positive integer $H$ we have
\[
\#X^{\tr}(\QQ, H)\leq cH^\varepsilon.
\]
\end{theorem}

Here $X^\tr$ denotes the \emph{transcendental part of $X$}, obtained from $X$ by removing all connected semi-algebraic curves from $X$, and $X^\tr(\QQ, H)$  is the set of rational points $(a_1/b_1, \ldots, a_n/b_n)\in X^\tr\cap \QQ^n$ for which $|a_i|, |b_i|\leq H$.
This theorem has seen some spectacular applications in Diophantine geometry and Hodge theory, and has recently led to the resolution of the Andr\'e--Oort conjecture~\cite{Andre-Oort}.

In the non-Archimedean setting, a Pila--Wilkie type counting theorem was first obtained by Cluckers--Comte--Loeser~\cite{CCL-PW} for $p$-adic fields with an analytic structure, and a uniform version was later proven by Cluckers--Forey--Loeser~\cite{CFL}, leading to results also in large positive characteristic.
In upcoming work with Cluckers and Halupczok, we will extend these results to $p$-adic fields equipped with $1$-h-minimal structure~\cite{CHV}.
This $1$-h-minimality is an axiomatic framework for tame geometry in valued fields, similar to o-minimality for real closed fields~\cite{CHR, CHRV}.

The aim of this note is to prove a Pila--Wilkie theorem in arbitrary $1$-h-minimal valued fields, and not just in $p$-adic fields.
A similar result appears in~\cite{CNSV}, where we proved a Pila--Wilkie theorem for equicharacteristic zero valued fields.
However, there one uses a different notion of rational points, and one bounds them by providing a finite-to-one map from the rational points to a higher residue ring.
In contrast, the results in this paper bound the number of rational points purely in terms of the residue characteristic, and are insensitive to the residue field and value group.
Unfortunately the methods only work for curves, and we are unable to obtain results in higher dimensions.
In~\cite{CHV} we originally expected to be able to prove this type of result in any dimension, but the higher-dimensional case seems more difficult than expected.

\subsection*{Main result}
Let $K$ be a valued field of mixed characteristic $(0,p)$, equipped with $1$-h-minimal structure.
Whenever we say \emph{definable}, we will mean definable in this $1$-h-minimal structure with parameters from $K$.
Similar to o-minimality, definable sets in $K^n$ have a dimension with the usual properties.
By a \emph{curve} in $K^n$ we then simply mean a one-dimensional definable set.
A curve $C\subset K^n$ is said to be \emph{transcendental} if for every algebraic curve $D\subset K^n$ the intersection $D\cap C$ is finite.
For a positive integer $H$, denote by $C(\QQ, H)$ the set of rational points $(a_1/b_1, \ldots, a_n/b_n)\in C\cap \QQ^n$ for which $|a_i|, |b_i| \leq H$ for $i=1, \ldots, n$.
In other words, this is simply the set of rational points on $C$ of height at most $H$.
Our main result is then as follows. 

\begin{theorem}\label{thm:main}
Let $K$ be a valued field of mixed characteristic $(0,p)$ equipped with $1$-h-minimal structure.
Let $C\subset K^n$ be a transcendental definable curve, then for every $\varepsilon > 0$ there exists a constant $c = c_\varepsilon>0$ such that for every $H\geq 1$ we have
\[
\# C(\QQ, H)\leq cH^\varepsilon.
\]
\end{theorem}

As an example, this result applies to $K=\CC_p$ or $K=\QQ_p^{\mathrm{unram}}$ equipped with an analytic structure, and is the first such theorem in arbitrary valued fields.

Similarly to the Pila--Wilkie theorem, this result has several natural generalizations.
For example, by compactness one can deduce a version for definable families $(C_y)_{y\in Y}$ of transcendental curves $C_y\subset K^n$ where the constant $c$ is independent of the parameter $y\in Y$, and only depends on the family $(C_y)_y$.
One can also prove a uniform-in-$p$ version of this result, similar to~\cite{CFL}.
More precisely, the proof of Theorem~\ref{thm:main} shows that $c$ is of the form $Mp^\alpha$, where $M$ and $\alpha$ are positive integers which are even independent of $K$.
This allows one to prove Theorem~\ref{thm:main} also for valued fields of large positive characteristic, where the characteristic required depends on $C$.
As this paper is intended to be short and accessible, we leave these generalizations to the interested reader.

Other generalizations, for example counting algebraic points of bounded degree, or versions with blocks as in~\cite{Pila.algebraic} are more difficult to obtain.
The reason is that these rely on a higher-dimensional version of Theorem~\ref{thm:main}.
We will come back to issues in generalizing this theorem to higher dimension at the end of Section~\ref{sec:determinant.method}.

\subsection*{Proof strategy} Let us briefly discuss the proof of Theorem~\ref{thm:main}.
So let $C\subset \cO_K^n$ be a transcendental definable curve.
\begin{enumerate}
\item After a projection argument we may assume that $n=2$. 
This is achieved in Lemma~\ref{lem:proj}.
\item We then write $C$ as the union of finitely many graphs of functions $f: U\subset \cO_K\to \cO_K$. 
By 1-h-minimality, we may assume that these functions are $C^r$ and are well-approximated by their Taylor polynomial.
Combining this with the Jacobian property from~\cite{CHRV}, we moreover obtain strong bounds on the derivatives of $f$.
In particular, if $B\subset U$ is a suitable open ball and $s: \cM_K\xrightarrow{\sim} B$ is a scaling map, then we show that $f\circ s$ is $T_r$ in the sense of~\cite{CCL-PW, CFL, CHV}.
This avoids the usage of $r$-th power maps and is similar in spirit to the resolution of Wilkie's conjecture due to Binyamini--Novikov--Zak~\cite{BinNovZak}.
\item We now apply the determinant method similarly to Bombieri--Pila~\cite{Bombieri-Pila}, which gives us a radius $\rho$ such that for every ball $B\subset \cM_K$ of radius $\rho$, there exists an algebraic curve passing through all rational points of the graph of $f\circ s$ over $B$.
\item The final ingredient is to show that most balls $B\subset \cM_K$ of radius $\rho$ do not actually yield any rational points at all, and these may safely be ignored in the counting.
We then count the remaining balls of radius $\rho$ and show that there are at most $O(H^\varepsilon)$ of these.
This leads to the improvements of the current paper, and the results valid in every valued field.
\end{enumerate}

\subsection*{Acknowledgements} The author thanks Raf Cluckers and Mathias Stout for many conversations about rational points and parametrizations, especially in higher dimensions. 
The author is supported by the Humboldt foundation.

\section{Preliminaries}

Throughout this entire paper we will work in a fixed valued field $K$ of mixed characteristic $(0,p)$ in some language $\cL\supset \cL_\val$.
We will assume that $\Th_{\cL}(K)$ is $1$-h-minimal, see~\cite{CHRV} for details.
The precise definition is not important for us, and we will recall the required consequences of $1$-h-minimality in this section.

We will moreover assume that $\Th_{\cL}(K)$ has algebraic Skolem functions, i.e.\ every $\emptyset$-definable finite-to-one surjection $X\to Y$ between $\emptyset$-definable sets has a $\emptyset$-definable section.
Or in other words, for every $L\equiv_\cL K$ and every $A\subset L$ we have that $\acl_{\cL}(A) = \dcl_{\cL}(A)$.
This is no loss in generality in view of~\cite[Prop.\,3.2.3]{CHRV}, together with the fact that Theorem~\ref{thm:main} only talks about rational points.

The valuation ring of $K$ is denoted by $\cO_K$, the maximal ideal by $\cM_K$ and the valuation group by $\Gamma^\times$.
We use multiplicative notation for the valuation $|\cdot |: K\to \Gamma = \Gamma^\times \cup \{0\}$.
For $N$ a positive integer let $\RV_N = K^\times / (1+N\cM_K)\cup \{0\}$, with corresponding quotient map $\rv_N: K\to \RV_N$ extended by $\rv_N(0) = 0$.
Write $\RV_N^\times = \RV_N\setminus \{0\}$.

If $a\in K$ and $\lambda\in \Gamma^\times$ then the open ball around $a$ with radius $\lambda$ is denoted by $B_{<\lambda}(a)$.
The closed ball is denoted by $B_{\leq \lambda}(a)$.
If $B$ is an open ball then its radius is denoted by $\radop B$.
For $c\in K$ and $N$ a positive integer, an open ball $B\subset K$ is said to be \emph{$N$-next to $c$} if there exists a $\xi\in \RV_N^\times$ such that $B = \rv_N^{-1}(\xi) + c$.
If $C\subset K$ is a finite set, then an open ball $B\subset K$ is \emph{$N$-next to $C$} if it is of the form $\cap_{c\in C} B_c$ where each $B_c$ is $N$-next to $c$.
Note that two elements $x,y\in K\setminus C$ lie in the same ball $N$-next to $C$ if and only if for every $c\in C$ we have $|x-y| < |N||x-c|$.

Definable functions in $1$-h-minimal structures automatically have the Jacobian property.

\begin{theorem}[Jacobian property]\label{thm:jac.prop}
Let $f: K\to K$ be a $\emptyset$-definable function and let $M$ be a positive integer.
Then there exists a finite $\emptyset$-definable set $C\subset K$ and a positive integer $N$ such that for each ball $B$ $N$-next to $C$ the following hold:
\begin{enumerate}
\item $f$ is $C^1$ and $\rv_M(f')$ is constant on $B$,
\item for all $x,y\in B$ we have
\[
\rv_M(f(x)-f(y)) = \rv_M(f'(x)(x-y)),
\]
\item if $B'\subset B$ is an open ball then $f(B')$ is either a singleton or an open ball of radius $|f'(x)|\radop B'$ (for any $x\in B$).
\end{enumerate}
\end{theorem}

\begin{proof}
See~\cite[Cor.\,3.1.4]{CHRV}.
\end{proof}

If the above properties hold for $f$ and an open ball $B$ contained in the domain of $f$, then we say that $f$ has the \emph{Jacobian property on $B$}.

If $f: U\subset K\to K$ is a $C^r$ function and $y\in U$ then we define the \emph{order $r$ Taylor polynomial} as usual via
\[
T_{f,y}^{\leq r}(x) = T_{f,y}^{<r+1}(x) = \sum_{i=0}^r \frac{f^{(i)}(y)}{i!} (x-y)^i.
\]
Definable functions are well-approximated by their Taylor polynomial.

\begin{theorem}[Taylor approximation]\label{thm:Taylor}
Let $f:  K\to K$ be a $\emptyset$-definable function and let $M$ be a positive integer.
Then there exists a finite $\emptyset$-definable set $C\subset K$ and a positive integer $N$ such that for every ball $B$ $N$-next to $C$, $f$ is $C^{r+1}$, $|f^{(r+1)}|$ is constant on $B$, and for $x,y\in B$ we have
\[
|f(x) - T^{\leq r}_{f,y}(x)| = \left| \frac{1}{(r+1)!}f^{(r+1)}(y)(x-y)^{r+1}\right|.
\]
\end{theorem}

\begin{proof}
See~\cite[Thm.\,3.1.2]{CHRV}.
\end{proof}

If the results of this theorem hold for $f$ on an open ball $B$ contained in its domain, then we say that $f$ has \emph{order $r$ Taylor approximation on $B$.}

The following notion of $T_r$ maps is key to the Diophantine application.

\begin{definition}[$T_r$ maps]
Let $r\geq 0$ be an integer and let $f: U\subset \cO_K\to \cO_K$ on an open $U$.
We say that $f$ is \emph{$T_r$} if $f$ is $C^r$ on $U$, $|f^{(i)}(x)|\leq |i!|$ for $x\in U$ and $i=0, \ldots, r$, and for every $x,y\in U$ we have
\[
|f(x)-T^{<r}_{f,y}(x)| \leq |x-y|^r.
\]
\end{definition}

\section{Parametrizations}

We first show that we may assume that $C$ is a planar curve.

\begin{definition}
Let $d$ be a positive integer.
A subset $X\subset K^2$ is said to be \emph{non-algebraic up to degree $d$} if for every algebraic curve $D\subset K^2$ of degree at most $d$, $D\cap X$ is finite.
\end{definition}

If $X$ is definable in a $1$-h-minimal structure and non-algebraic up to degree $d$ then $D\cap X$ is in fact uniformly bounded over all algebraic curves of degree at most $d$.
Indeed, this follows from uniform finiteness as in~\cite[Lem.\,2.5.2]{CHR}.

\begin{lemma}\label{lem:proj}
Let $C\subset K^n$ be a transcendental definable curve and let $d$ be a positive integer.
Then there exists a partition $C = \sqcup_i C_i$ into $\emptyset$-definable subsets $C_i\subset C$, and for each $i$ a coordinate projection $\pi_i: K^n\to K^2$ such that $\pi_{|C_i}$ is a bijection onto its image, and $\pi_i(C_i)$ is non-algebraic up to degree $d$.
\end{lemma}

\begin{proof}
See~\cite[Lem.\,4.2.2]{CNSV}.
\end{proof}

Next, we show that we can parametrize $X$ by finitely many graphs of functions.
Denote by $\pi: K^2\to K$ the projection onto the first coordinate.

\begin{lemma}\label{lem:preparam}
Let $X\subset \cO_K^2$ be a $\emptyset$-definable curve and let $r$ be a positive integer.
Assume that $\pi$ is finite-to-one when restricted to $X$.
Then there exists a positive integer $M$, finitely many $\emptyset$-definable functions $f_i: U_i\subset \cO_K\to \cO_K$ and $\emptyset$-definable elements $c_i\in \cO_K$ such that $X$ is the union of the graphs of the $f_i$, such that $U_i$ is $M$-prepared by $c_i$ and such that the following hold for every $j=0, \ldots, r+1$ and every ball $B\subset U_i$ $M$-next to $c_i$:
\begin{enumerate}
\item $f_i$ is $C^{r+1}$ on $B$,
\item $\rv(f_i^{(j)})$ is constant on $B$,
\item $f_i^{(j)}$ has the Jacobian property on $B$,
\item $f$ has order $r$ Taylor approximation on $B$, and
\item for $x\in B$ we have
\[
|f_i^{(j)}(x)|\leq \frac{1}{\radop (B)^j}.
\]
\end{enumerate}
\end{lemma}

\begin{proof}
By cell decomposition as in~\cite[Thm.\,3.3.3]{CHRV} we can indeed write $X$ as a union of graphs of functions $f_i: U_i\to \cO_K$.
Note that this uses the existence of algebraic Skolem functions.
Theorems~\ref{thm:jac.prop} and~\ref{thm:Taylor} show that we may assume the $f_i$ to be $C^{r+1}$, $\rv(f_i^{(j)})$ to be constant on $B$, $f_i^{(j)}$ to have the Jacobian property, and $f$ to have order $r$ Taylor approximation.
It remains to check the last inequality, for which we reason by induction on $j$.

For $j=0$ we have that $|f_i(x)|\leq 1$, simply because $f_i$ lands in $\cO_K$.
Assume now that the inequality holds for $j$.
By the Jacobian property and since $\rv(f_i^{(j)})$ is constant on $B$, $f_i^{(j)}(B)$ is then contained in an open ball of radius at most $1/\radop(B)^j$.
On the other hand, the Jacobian property also tells us that $f_i^{(j)}(B)$ is an open ball of radius $|f_i^{(j+1)}(x)|\radop B$, for any $x\in B$. 
Hence indeed
\[
|f_i^{(j+1)}(x)|\leq \frac{1}{\radop (B)^{j+1}}.\qedhere
\]
\end{proof}

\begin{definition}
Let $r$ be a positive integer, $c\in \cO_K$ be $\emptyset$-definable, let $U\subset \cO_K$ be $M$-prepared by $c$ for some positive integer $M$ and let $f: U\subset \cO_K\to \cO_K$ be $\emptyset$-definable. 
Call $f$ an \emph{$r$-parametrizing map with centre $c$ and integer $M$} if $f$ is $C^r$, $f$ satisfies order $r-1$ Taylor approximation on balls $M$-next to $c$ for $i\leq r$, and for $x\in U, i\leq r$ we have
\[
|f^{(i)}(x)|\leq \frac{1}{|M|^i|x-c|^i}.
\]
\end{definition}

With this terminology, Lemma~\ref{lem:preparam} asserts that every definable $X\subset \cO_K^2$ for which $\pi_{|X}$ is finite-to-one is a finite union of graphs of $r$-parametrizing maps, for every $r$.

In~\cite{CCL-PW, CFL, CHV}, higher-dimensional analogues of Lemma~\ref{lem:preparam} are used in combination with power substitutions to prove that every definable $X\subset \cO_K^n$ admits a \emph{$T_r$-parametrization}.
In more detail, this means that there are finitely many $T_r$ maps $f_i: U_i\subset \cO_K^{\dim X}\to X$ which together cover $X$.
One can subsequently use these parametrizations to prove counting results for integral or rational points on transcendental definable sets.

Inspired by~\cite{BinNovZak}, to prove Theorem~\ref{thm:main} we do not actually need to find a $T_r$-parametrization of the entire set $X$, but rather only a $T_r$-parametrization which catches all rational points.
This insight allows us to replace power substitutions by scaling maps, and is key to proving Theorem~\ref{thm:main} for arbitrary valued fields.

Let $c\in K$ and let $M$ be a positive integer.
If $B$ is a ball $pM$-next to $c$ take $a\in B$.
Then we define the \emph{scaling map}
\[
s_a: \cM_K\to B: x\mapsto (a-c)(1+pMx)+c.
\]
Note that this map is a bijection.

\begin{lemma}\label{lem:scaling.Tr}
Let $f: U\subset \cO_K\to \cO_K$ be an $r$-parametrizing map with centre $c$ and integer $M$. 
Let $B\subset U$ be a ball $pM$-next to $c$ and take $a\in B$. 
Then the map $f\circ s_a: \cM_K\to \cO_K$ is $T_r$.
\end{lemma}

\begin{proof}
Since $f$ has Taylor approximation of order $r-1$, also $f\circ s_a$ will have Taylor approximation of order $r-1$.
Hence it is enough to check that $|(f\circ s_a)^{(i)}(x)|\leq |i!|$ for $x\in \cM_K$.
But this follows from the chain rule and the definition of a parametrizing map.
\end{proof}

\section{Determinant method}\label{sec:determinant.method}

We now show how to catch all rational points of bounded height on the graph of a parametrizing map in a small number of algebraic curves.

%

\begin{lemma}
Let $x_1, \ldots, x_m\in K$ be rational numbers, let $c\in \cO_K$, and let $N$ be a positive integer.
Assume that $|x_i-x_j|\geq |p^N||x_i-c|$ and that $|x_i-c| = |x_j -c|$ for all $i\neq j$.
Then $m\leq p^{N+1}$.
\end{lemma}

\begin{proof}
Let $\delta = |x_i-c|$.
Our assumption shows that all values $x_i-c$ are distinct modulo $p^{N+1}\delta$.
Hence for $i=1, \ldots, m$, the elements $x_i-x_1 = (x_i-c) - (x_1-c)$ are all rational numbers which are distinct modulo $p^{N+1}\delta$.
They also satisfy $|x_i-x_1|\leq \delta$.
Hence $m$ is bounded by $\#(\ZZ / p^{N+1}\ZZ) =  p^{N+1}$ as desired.
\end{proof}

\begin{lemma}\label{lem:count.balls}
Let $c\in \cO_K$ and let $H, N,M$ be positive integers.
Then there are at most $2Mp^{N+1}(1+2\log_p (H))$ balls $B\subset U$ which are $|Mp^N|$-next to $c$ and which contain a rational points of height at most $H$.
\end{lemma}

\begin{proof}
Let $S\subset \QQ$ be the set of rational numbers of height at most $H$.
We have to count the size of $\rv_{Mp^N}(c-S)$.
Let $y$ be the element in $S$ for which $|c-y|$ is minimal.
Then for every $x\in S$ the triangle inequality shows that either $|c-y|\leq |x-y|=|c-x|$ or $|x-y|\leq |c-y|=|c-x|$.
In any case, $|c-x|$ is either equal to $|c-y|$ or to $|x-y|$, and so there are at most $2+4\log_p (H)$ values $|c-x|$ for $x\in S$.

Fix a value $\delta = |y-c|$ for some $y\in S$.
Let $x_1, \ldots, x_m\in S$ be such that $|x_i - c| = \delta$ for all $i$ and such that the values $\rv_{Mp^N}(x_i-c)$ are all distinct.
This means that for $i\neq j$ we have $|x_i-x_j|\geq |Mp^{N}||x_i-c|$.
Then the previous lemma shows that $m\leq Mp^{N+1}$, as desired.
\end{proof}

For the determinant method we need the following determinant estimate.

\begin{lemma}[Determinant estimate]\label{lem:det.lemma}
Let $r\geq 1$ be an integer and define $e = \binom{r}{2}$.
Let $f_i: B\subset \cO_K\to \cO_K$ be $T_r$ on an open ball $B$ of radius $\lambda$, for $i=1, \ldots, r$.
For $x_1, \ldots, x_r\in B$ we then have
\[
|\det (f_i(x_j))_{i,j}|\leq \lambda^e.
\]
\end{lemma}

\begin{proof}
Since $f_i$ is $T_r$ we have for every $x_j$ that
\[
f_i(x_j) = T^{<r}_{f_i, x_1}(x_j) + \alpha_{i,j}(x_j-x_1)^r,
\]
for some $\alpha_{i,j}\in \cO_K$.
For integers $1\leq i,j\leq r$ and $0\leq k\leq r-1$ define $\beta_{i,j,k} = f_i^{(k)}(x_1)/k!$ and $\beta_{i,j,r} = \alpha_{i,j}$.
Note that for $k\leq r-1$, $\beta_{i,j,k}$ is in fact independent of $j$.
Also, $|\beta_{i,j,k}|\leq 1$ by definition of being $T_r$.

By expanding the determinant, we may write $\det (f_i(x_j))_{i,j}$ as a sum of determinants $\Delta_\ell$ where $\ell = (\ell_2, \ldots, \ell_r)\in \NN^{r-1}$ is such that $\ell_i\leq r$.
Each $\Delta_\ell$ is of the form
\[
\Delta_\ell = \begin{vmatrix}
f_1(x_1) & \beta_{1,2,\ell_2} (x_2-x_1)^{\ell_2} & \cdots & \beta_{1,r, \ell_r} (x_r-x_1)^{\ell_r} \\
\vdots & & \ddots & \\
f_r(x_1) & \beta_{r, 2, \ell_2} (x_2-x_1)^{\ell_2} & \cdots & \beta_{r, r, \ell_r} (x_r-x_1)^{\ell_r}
\end{vmatrix}.
\]
Now, if there are $a\neq b$ with $\ell_a = \ell_b< r$ then the $a$-th column and $b$-th column of this matrix are linearly dependent, and hence $\Delta_\ell = 0$ in that case.
So if $\Delta_\ell$ is non-zero then necessarily $\sum_j \ell_j \geq \sum_{j=1}^{r-1} j = \binom{r}{2}=e$.
From each column of $\Delta_\ell$ we can then factor out $(x_j-x_1)^{\ell_j}$ and so we see that
\[
|\Delta_\ell|\leq \prod_{j=2}^\mu |x_j-x_1|^{\ell_j}\leq \lambda^e.
\]
We conclude by the ultrametric triangle inequality.
\end{proof}

\begin{lemma}\label{lem:catch.int}
For every integer $d\geq 1$ there are an integer $r$ and positive constants $\varepsilon = \varepsilon(d), m>0$ such that the following holds.
Let $c\in K$ and let $f: U\subset \cO_K\to \cO_K$ be an $r$-parametrizing map with centre $c$ and integer $M$.
Then for every integer $H\geq 1$ the set $\graph(f)(\QQ, H)$ is contained in at most
\[
m MH^\varepsilon
\]
algebraic curves of degree at most $d$.
Moreover, $\varepsilon\to 0$ as $d\to \infty$.
\end{lemma}

For an integer $b$ we denote by $|b|_\RR$ its usual absolute value, to distinguish it from the norm on $K$.

\begin{proof}
Define $r = \binom{d+2}{2}$ and recall that $e = \binom{r}{2}$.
Let $B$ be a ball $pM$-next to $c$, take $a\in B$ and let $g = f\circ s_a: \cM_K\to \cO_K$, where $s_a$ is the scaling map.
By Lemma~\ref{lem:scaling.Tr} the map $g$ is $T_r$.
Hence also $s_a^i g^j$ is $T_r$ for any $i,j\in \NN$.

Let $B'\subset \cM_K$ be an open ball of radius $|p^N|$, for some positive integer $N$ to be determined later, and let $x_1, \ldots, x_r\in B'$ be such that $s_a(x_i)$ and $g(x_i)$ are rational points of height at most $H$.
Consider the determinant $\Delta = \det(s_a(x_i)^j g(x_i)^k)_{1\leq i\leq r, 0\leq j+k\leq d}$.
We will prove that for a suitable choice of $N$, $\Delta$ becomes $0$.

For every $i$, there exists a positive integer $b_i\leq H^2$ for which $b_i s_a(x_i)g(x_i)$ is an integer, and consider the determinant $\Delta' = \det(b_i^d s_a(x_i)^j g(x_i)^k)_{1\leq i\leq r, 0\leq j+k\leq d}$.
Since all entries of $\Delta'$ are positive integers bounded in absolute value by $H^{3d}$, $\Delta'$ is an integer of absolute value at most $r!H^{3rd}$.
Note that we also have $\Delta' = \prod_i b_i^d \Delta$.
By Lemma~\ref{lem:det.lemma} we have $|\Delta|\leq |p^{Ne}|$, and so also $|\Delta'|\leq |p^{Ne} \prod_i b_i^d|\leq |p^{Ne}|$.
If $\Delta$ is non-zero, then since $\Delta'$ is an integer we obtain that
\[
p^{Ne} \leq |\Delta'|_\RR \leq r!H^{3rd}.
\]
Upon taking
\[
N > \frac{\log_p (r!) + 3rd \log_p H }{e}
\]
we obtain a contradiction and so $\Delta$ must be zero.

By linear algebra there now exists for each open ball $B'\subset \cM_K$ of radius $|p^N|$ an algebraic curve of degree at most $d$ passing through all rational points of $\graph(f_{|s_a(B')})$ of height at most $H$.
By Lemma~\ref{lem:count.balls} we only have to consider $2Mp^{N+2}(1+2\log_p (H))$ of these open balls to catch all rational points of height at most $H$ on the graph of $f$.
Hence we conclude that $\graph(f)(\QQ, H)$ is contained in at most
\[
2Mp^{N+2}(1+2\log_p (H)) \leq 4M p^3 r! H^{\frac{3rd}{e}}(1+2\log_p H) 
\]
algebraic curves of degree at most $d$.
Now define $\varepsilon = 6rd/e$, then there exists a constant $m>0$ depending only on $d$ and $p$ such that 
\[
4M p^3 r! H^{\frac{3rd}{e}}(1+2\log_p H) \leq mMH^\varepsilon.
\]
For the final statement, note that $rd \ll d^3$ while $e\gg d^4$, and hence $\varepsilon\to 0$ as $d\to \infty$.
\end{proof}

We now have all the ingredients for the proof of Theorem~\ref{thm:main}.

\begin{proof}[Proof of Theorem~\ref{thm:main}]
Take $d\geq 1$ such that $\varepsilon(d) < \varepsilon$, where $\varepsilon(d)$ is the value from Lemma~\ref{lem:catch.int}.
By Lemma~\ref{lem:proj} we may assume that $C\subset \cO_K^2$ is a planar curve which is non-algebraic up to degree $d$.

Let $\pi: \cO_K^2\to \cO_K$ be the projection onto the first coordinate.
Then $\pi$ is finite-to-one when restricted to $C$.
Indeed, otherwise $C$ would contain a subset of the form $\{a\}\times B$ with $a\in \cO_K$ and $B\subset \cO_K$ an open ball.
But then $C$ would not be non-algebraic up to degree $d\geq 1$.
By Lemma~\ref{lem:preparam} we may then write $C$ as a finite union of $r$-parametrizing maps, so it is enough to focus on one such $r$-parametrizing map $f$ with centre $c\in \cO_K$ and integer $M$.

Lemma~\ref{lem:catch.int} shows that $\graph(f)(\QQ, H)$ is contained in at most $mMH^\varepsilon$ algebraic curves of degree at most $d$.
The intersection of the graph of $f$ with an algebraic curve of degree at most $d$ is finite and uniformly bounded by some integer $N$.
Hence we conclude that $\graph(f)(\QQ, H)$ consists of at most 
\[
c' H^\varepsilon
\]
points, where $c' = mNM$ depends only on $C$.
\end{proof}

To end this paper, let us discuss the main obstacles in extending Theorem~\ref{thm:main} to higher dimensions.
Many of the ingredients used in this paper already work well in higher dimensions, such as constructing parametrizations as in Lemma~\ref{lem:preparam} or the determinant estimate from Lemma~\ref{lem:det.lemma}.
However, it is not at all clear to us how to generalize Lemma~\ref{lem:count.balls}, which is a key component for the proof of Lemma~\ref{lem:catch.int}.

Let us be more precise. 
Suppose that we wish to prove Theorem~\ref{thm:main} in dimension $2$, say for the graph of some definable function $f: U\subset \cO_K^2\to \cO_K$.
Using cell decomposition, we can write $U$ as a union of twisted boxes, which are of the form
\[
B = \{(x,y)\in \cO_K^2\mid \rv_M(x-c_1) = \xi_1, \rv_M(y-c_2(x)) = \xi_2\}
\]
for some $\xi_1,\xi_2\in \RV_M$, $c_1\in \cO_K$ and definable function $c_2: \cO_K\to \cO_K$.
For $a\in B$ there is a similar scaling map $s_a: \cM_K^2\to B$ for which $f\circ s_a$ is $T_r$.
The issue is now that it is not clear how to bound the number of twisted boxes $B\subset U$ one has to consider to catch all rational points of bounded height on $X$.
In other words, it is not clear how to generalize Lemma~\ref{lem:count.balls}.

\bibliographystyle{amsplain}
\bibliography{anbib}

\end{document}